\documentclass{article}
\usepackage{tikz}
\usepackage{url}
\usepackage{color}
\usepackage{algorithm}
\usepackage{algorithmic}

\usetikzlibrary{arrows}
\pgfarrowsdeclarecombine{twotriang}{twotriang}%
{latex}{}{}{latex}

\usepackage[T1]{fontenc}
\usepackage[utf8]{inputenc}
\DeclareUnicodeCharacter{00A0}{~}

\usepackage{amsmath,amsthm,amssymb}

\DeclareMathOperator{\LeastN}{Ln}
\DeclareMathOperator{\bound}{\partial}

\theoremstyle{plain}
\newtheorem{assumption}{Assumption}
\newtheorem{proposition}{Proposition}
\newtheorem{remark}{Remark}

\newcommand{\RR}{{\mathbb R}}

\newcommand{\redd}[1]{}

\title{Pricing Vehicle Sharing with Proximity Information}

\author{
Jakub Mare{\v c}ek$^{1}$\thanks{\tt {jakub.marecek@ie.ibm.com}}, Robert Shorten$^{12}$, Jia Yuan Yu$^{3}$\\[3mm]
$^{1}$ IBM Research, Ireland \\
$^{2}$ University College Dublin, Ireland \\
$^{3}$ Concordia University, Canada
}

\begin{document}

\maketitle

\begin{abstract}
For vehicle sharing schemes, where drop-off positions are not fixed, 
we propose a pricing scheme, 
where the price depends in part on the distance between where a vehicle is being dropped off and where the closest shared vehicle is parked.
Under certain restrictive assumptions, we show that this pricing leads to a socially optimal spread  of the vehicles within a region.
\end{abstract}


\section{Introduction}

The numbers of cars on the roads seem ever increasing, with the consequent 
issues of congestion on the roads, lack of parking space, and pollution.
Car-sharing schemes, which let members use any car in a fleet, 
promise to reduce the issues considerably \cite{millard2005car,firnkorn2011will,Babbage}.
Until the penetration of car sharing in any given area is high enough, however, 
 the distance to travel to pick up a car may be prohibitive.

In car-sharing systems, where the pick-up and drop-off have to be the same, e.g. ZipCar, 
 one may optimise the positions and fix them.
In the so called floating-car systems, e.g. car2go \cite{firnkorn2011will} and Autolib \cite{Autolib}, 
  one may use the shared car for a one-way journey, often without reservations.
Although such flexible, open-ended use is preferable from the users' point of view \cite{schwieger2004international,millard2005car},
  it makes the system more difficult to operate, because of the costs of balancing the spatial distribution of cars,
 which are much higher than in bicycle-sharing schemes.
We hence focus on the question as to how to incentivize the drivers to spread the 
cars more evenly.


It seems appropriate that a component of the price paid for sharing the car 
 should be based on where you drop off the car.
Clearly, one could make the price proportional to the inverse 
  of the distances between where the car is dropped off and each other car.
That, however, would require knowing the positions of all other cars, 
  violating the privacy of their drivers.
We hence propose two approximations, which use use only the set of distances of $k$ cars parked closest to the car dropped off and a distance to the boundary, both of which 
can be obtained using sensing or short-range wireless radio systems to communicate with parked cars in the vicinity.
Actually, in the simplest case, we propose that the component is proportional to the inverse 
  of the distance between where the car is dropped off and the 
 closest car already parked.
In further cases, more elaborate functions and a small number of already parked cars are considered.
Overall, we show that even the use of such local approximations in the pricing leads to a uniform distribution of the cars across a region, under certain assumptions.

Specifically, we study two models of dynamics, where drivers try to minimize the fee incurred, considering the pricing scheme above.
Our main findings are:
\begin{itemize}
\item The socially optimal positions of cars within an polygonal region, 
  as well as the best possible outcome of policies that minimize fees incurred, 
  can be computed.
\item In a certain synchronous approximation of the asynchronous dynamics, the relative positions of cars, whose drivers minimise the proposed fees, provably converge to a limit that is socially optimal. 
\item For other pricing functions considered, the relative positions may converge to a limit, whose cost is far from the social optimum.
\end{itemize}

Notice that these are only first steps towards pricing with full consideration of the complex environment of modern cities, with non-convex road systems, one-way traffic restrictions, and non-uniform and time-varying demands.
For self-driving cars, whose positions can keep changing even after drop-off, 
  the model seems to be rather realistic.
For frequent enough drop-offs, 
  even the behaviour of humans dropping the cars off,
  as incentivized by the pricing scheme,
  may be well-approximated by the model.
Our first results need validation using proper planning techniques, but possibly open up
a large area of research on the interface of revenue management and control.

\section{The Model} 

We consider a finite number $N$ of shared cars. 
Let $Q \subseteq \RR^2$ denote the region where these cars can travel, e.g., the 
metropolitan area the car-sharing scheme is restricted to.
Let $\bound(Q) \subseteq R^2$ denote the boundary of $Q$.
The positions of parked cars are publicly known.
For conveninence, we denote the cars of known positions by integers $K
= \{1, 2, \ldots k\}$.
For a continuous time index $t\in[0,\infty)$
Let $x^u(t) \in Q$ be the position of car $u \in K$ at time $t$.
We write $x(t) = (x^1(t) , \ldots, x^k(t))$ to denote all known position of 
cars at time $t$.
In discrete-time models, for $n = 1, 2, \ldots$,
we use $x^u_n \in Q$ for the position of car $u \in K$ at time $n$ and
$x_n \in Q^k$ to denote the corresponding vector of known positions at time $n$.

For both discrete and continuous time,
we define a common notion of \emph{inconvenience} $U^*(x, u)$ of car $u \in K$ parked at position
$x^u \in Q$ relative to the aggregate position $x \in Q^{k}$:
\begin{align}
\max \left\{ \frac{1}{ \min_{y \in \bound(Q)} ||y - x^u||}, \frac{2}{\min_{v \in K, v \not = u} ||x^u - x^v||} \right\},
\label{eq:pair-wise-sum}
\end{align}
where $\min_{y \in \bound(Q)} ||y - x^u||$ is the distance from $x^u$ to the boundary of $Q$ and
where $|| \cdot ||$ denotes the Euclidean norm.
We think of it as the inconvenience cost of the location of car $u$ to 
the potential participants of 
the shared car scheme. Intuitively, parking a car $u$ near another shared car $v$
incurs a ``redundancy'' cost for not capturing the demand of other
potential customers elsewhere.

\begin{remark}
Although the definition of inconvenience may be seem non-obvious at first, it has some appealing properties.
First, the range of $U^*(x, u)$ is $[0, \infty)$.
Second, it is non-decreasing in the distance to the nearest other car.
Alternatives to inconvenience cost \eqref{eq:pair-wise-sum} such as $-
\min_{v\neq u} d(x^u, x^v)$ for a metric $d$ are not considered in this paper.
\end{remark}

We define the social cost of an aggregate position $x$ as:
\begin{align}
C(x) := \max_{1 \le u \le K} U^*(x, u)
\label{eq:social-cost}
\end{align}
and the cost at the social optimum:
\begin{align}
C^* := \min_{x} C(x).
\label{eq:social-opt}
\end{align}

We propose pricing models, where agent $u$ pays a fee $U(x, u)$ for
dropping their car at $x^u$. This $U(x, u)$ approximates the
inconvenience $U^*(x, u)$, while being 
computable using local data.
The specific functions are described later (c.f. Eq. \ref{eq:neigh} on page \pageref{eq:neigh}).

\subsection*{A Dynamical Model in Discrete Time}

Ultimately, we want to study an asynchronous dynamical system, 
where users respond to the price by changing the positions, where they drop off their car.
This can be seen as a discrete-time problem, with time discretised by the points in time, when a user drops the car off.
Further, we assume the time when a user drops the car off to coincide with a time, when another user picks up a car.

Formally, let us have a function $p_n: \mathbb{Z} \to \{ 1, 2, \ldots, k \}$, which could take the elements of the image in a cyclic order, for example.
Some further examples will be studied experimentally on page \pageref{sec:experimental} below.
Let us also have function Step, which computes the set of minimisers $\arg \min_{x^u_n \in Q} W(x, p_n)$,
picks one,
and scales the difference with respect to $x_n$
uniformly down across all coordinates as little as possible,
such that the Euclidean norm is less than or equal to ${s_\textrm{max}}$,
as detailed in Algorithm~\ref{alg:Step}.

Now, consider a model, where one parked car $p_n$ changes its position from $x^{p_n}_n$ to $x^{p_n}_{n+1}$ at every time step $n \in 1, 2, \ldots$ and the vector of position $x_n$ evolves as:
\begin{align}
x_{n+1} = \begin{cases}
\; x^{k}_n - \textrm{Step}^U_{s_\textrm{max}}\left( { x, p_n } \right), & \textrm{if } k = p_n \\
\; x^{k}_{n} & \textrm{otherwise}
\end{cases}
\label{eq:dynamics2}
\end{align}
where function Step is detailed in Algorithm~\ref{alg:Step}. 
Notice that the cases in \eqref{eq:dynamics2} can be seen as demultiplexing with selector $p_n$.
Further, we define $x_n = x$ to be a fixed-point if and only if $x_{n+1} = x$.

 \begin{algorithm}[t!]
 \begin{algorithmic}[1]
 \item[\bf Input:]
  $x \in \RR^{2k}$ are the $k$ known positions of parked cars in 2D,  
  $p \in K$ is the index of the car dropped off, and 
  ${s_\textrm{max}}$ is the bound on the step-length.
 \STATE $M \gets \arg \min_{x \in Q} U(x, p_n), x^{p_n}_n$, \\ such that only $x^u_n$ changes
 \STATE $m \gets \arg \min_{m \in M} || m || $, \\ which is unique by convexity arguments
 \STATE $D \gets m - x$, where the difference is element-wise
 \STATE find the smallest $c \ge 1$ such that $|| S || / c \le s_\textrm{max}$
 \RETURN $S / c$
 \end{algorithmic}
 \caption{$\textrm{Step}^U_{s_\textrm{max}}\left( { x, p_n } \right)$ function of the discrete-time dynamics \eqref{eq:dynamics2}.}
 \label{alg:Step}
\end{algorithm}


\subsection*{A Continuous-Time Approximation}

We also consider a continuous-time, synchronous model, based on 
\cite{Cortes2005}.
There, all cars move in at the same time, so as to minimise some fee $U(x, u)$ to be paid, which is a function 
of the positions $x$ an the agent $u$, which approximates $U^*(x, u)$.
The dynamics can be formulated as:
\begin{align}
\frac{d}{dt} x(t) & = - \LeastN\left( \partial \left( 
\arg \min_{x \in Q} U(x, u) 
\right) \right) (x(t)). 
\label{eq:dynamics}
\end{align}
where $\LeastN: 2^{\RR^{k}} \to \RR^{k}$ is a mapping that associates a subset
$S$ of $\RR^{k}$ with its least-norm element. 
For a locally Lipschitz function $f$, $\LeastN(\partial f) : \RR^{k} \to \RR^{k}$ is 
a generalized gradient vector field,
given by $x \to \LeastN(\partial f)(x)$. 
Although the $\arg \min_{x \in Q} U(x, u)$ is not guaranteed to be unique for a general $U$,
  Proposition 3.3 \cite{Cortes2005} suggests two functions $U$, for which the 
  $\partial \left( \arg \min_{x \in Q} U(x, u) \right)$ has a closed-form solution.
There, vector $-\LeastN(\partial f)(x)$ is a direction of descent and is guaranteed to
exist as per Theorem 2.4 of \cite{Cortes2005}.
Here, we define $x(t) = z$ to be a fixed point of the above dynamical system if and only if $\frac{d}{dt} x(t) = 0 \quad$.

\begin{remark}
  One should see the continuous-time model as a synchronous approximation of an asynchronous process, 
  in the limit of the number of users parking at one time,  
  and infinitesimally small time step.
  Alternatively, wihtin the calculus on measure chains \cite{hilger1990analysis}, 
  one could formulate both models as dynamics, which differ in the timescale,
  $\mathbb{T = Z}$ and $\mathbb{T = R}$, respectively, and the limit of the
  number of users parking at the same time, being $1$ and $k$, respectively.
  We do not attempt to answer the question of consistency of the approximation 
  of one model by another in this paper, other than in simulations for specific $p_n$,
  which seem encouraging.
\end{remark}


\subsection*{Neighbourhood-Based Pricing $V, W$}
\label{sec:adj}

In this section, we introduce two approximations of $U^*$. 
The approximations are motivated by the fact that computing $U^*$ and the social cost \eqref{eq:social-cost} using pairwise distances  
 and distances to the boundary, for each agent, may be difficult.
The driver may be reluctant to disclose their position, in order for the computation to be performed at a central server, thereby making it necessary to transfer 
 the positions of all parked cars to the driver, which may be time-consuming or costly or both.
The proposed approximation use only the set of distances $D(x,u)$ of $|D(x,u)|$ parked cars closest to $x^u$ and a distance to the boundary.
See a block diagram in Figure~\ref{fig:block}, 
with 
$k$-NN, proj, and V referring to Line 1 of Step and
demux referring to \eqref{eq:dynamics2}.

\begin{figure}[t]
\centering 
\includegraphics[clip=true,trim=1.5cm 0cm 0cm 0cm]{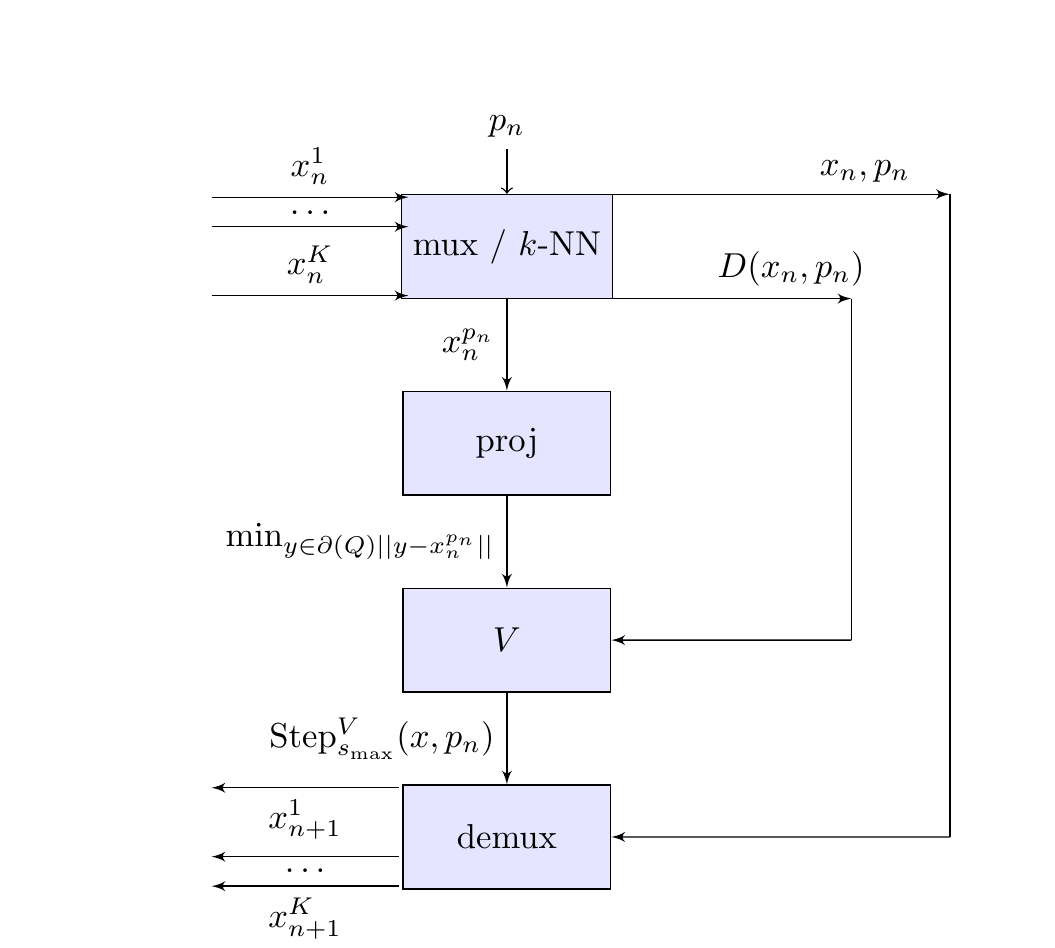}
\caption{A block-diagram for the neighbourhood-based pricing $V$.}
\label{fig:block}
\end{figure}

Two natural approximations of $U^*$ for pricing of dropping car $u$ considering positions $x$ are:
\begin{align}
V(x, u) & := 1/\min \{ \frac{1}{2} \min_{y \in \bound(Q)} ||y - x^u||, \min_{d \in D(x,u)} \{ d \} \}  \\
W(x, u) & := \left( \frac{1}{2} \min_{y \in \bound(Q)} ||y - x^u|| + \sum_{d \in D(x,u)} d \right) ^{-1}.
\label{eq:neigh}
\end{align}
It remains unclear, however, how do the above approaches perform, in theory.
Is one $V$ better than $W$? When? When do they perform as well as $U^*$ using the complete
information?

\section{An Analysis of the Dynamics in Discrete Time}

In order to answer these questions, we study the two models of the dynamics, in turn.
First, the dynamics in discrete time, i.e. 
with $x_n$ generated by \eqref{eq:dynamics2},
it may seem rather difficult to reason about the $\lim_{n \to \infty}
sx_n$. \redd{[this sentence too soft]}
Initially, we hence study the problem of best-possible positions of parked cars within a region,
at any time $n$. 
This corresponds to the problem of of setting the positions in systems, where the pick-up and drop-off have to be the same, e.g. ZipCar,
and provides a lower bound on the best possible outcome achievable using any pricing function $U$.
We make the following assumptions:

\begin{assumption}[Arbitrary Positions] 
At any time, the positions of all cars $u$ are distinct, but otherwise can be changed arbitrarily
within a convex compact polyhedron $Q \subseteq R^d$ with non-empty interior.
\label{as:Holonomic}
\end{assumption}

Then:

\begin{proposition}[Optimum Position]
Under Assumption
\ref{as:Holonomic}, the social optimum $C^*$ as well as the best possible outcome with fee $V$ or $W$,
\begin{align}
\min_{x \in Q} \sum_{u \in K} U^*(x, u), \\
\min_{x \in Q} \sum_{u \in K} V(x, u), \\
\min_{x \in Q} \sum_{u \in K} W(x, u), \\
\label{eq:SocialOptBasic}
\end{align}
can be approximated to any arbitrary precision.
\label{prop:socialopt1}
\end{proposition}

\begin{proof}[Proof of Proposition \ref{prop:socialopt1}]
Notice that the outcome is the same in both discrete and continous time.
The proof has two steps, which are the same for both discrete and continous time,
and relies heavily on the results of Bulgarin, Henrion, and Lasserre \cite{Bugarin2011}
and the observation that each of $U^*, V$, and $W$ is expressible by fractional programming, where both the variable and constraints are given by finite sums such as $\sum_i \frac{p_i(\cdot)}{q_i(\cdot)}$ and $p_i$ and $q_i$ are polynomials.

First, one needs to formulate an explicit fractional programming problem.
There are two minor complications. First, notice that the boundary of the polygon is defined by a number of line-segments. The distance of a point $x_0^u \in R^2$ to a line segment $L(a, b)$ connecting $a, b \in R^2$ can be computed by a closed form expression involving scalar $p = ((x_0 - a) \cdot (b - a))/||b - a||$, i.e.
\begin{align}
|| x_0^u, L(a, b)|| = 
\begin{cases}
\; ||x_0^u - a|| & \textrm{ if } p < 0 \\
\; ||x_0^u - b|| & \textrm{ if } p > 1 \\
\; ||x_0^u, a + p (b - a)|| & \textrm{ otherwise.}
\end{cases}
\end{align}
Next, notice that the taking of the minima, whose non-decreasing function is to be maximised, 
e.g. for pred or succ in the adjacency model or for the minimum of the distance to the boundary
and the pair-wise distances, 
does not change the problem profoundly, as it can be implemented by the addition of additional variables and constraints, e.g. 
$\max \min \{ a, b \} = \max c \text{ s. t. } c < a, c < b.$

The second step uses the technique of \cite{Bugarin2011}, who observe
that \begin{align}\min_x & \left\{ \frac{p(x)}{ q(x)} \; : \; x \in Q \right\}  \\ = \min_{\mu \in M(Q)} & \left\{ \int p(x) d \mu(x) \; : \; \int_{Q} q(x) d\mu(x) = 1 \right\}, \notag \end{align} where M(Q) is the space of finite Borel measures on 
$Q$.
From Assumption \ref{as:Holonomic}, the set $Q$ is compact and has non-empty interior. 
From the equivalence to a generalised problem of moments \cite{Lasserre2007}, it follows that
if $Q$ is compact and its interior is nonempty, then there exists 
a sequence of semidefinite programming (SDP) relaxations $R_r$ such that 
\begin{enumerate}
\item[(a) ]
$\inf R_r \to S^U \text{ as } r\rightarrow\infty.$
\item[(b) ]
There is no duality gap between $R_r$ and $R_r^{\star}$.
\item[(c) ]
If \eqref{eq:SocialOptBasic} has a unique global minimizer $x^{\star}\in Q$, one can extract $x^{\star}$ from $R_r$ for some $r$.
\end{enumerate}

\end{proof}

\begin{figure*}[!t]
\begin{align}
C & \left(\arg \min_{x \in Q} \max_{u \in K} \max \left\{ \frac{1}{ \min_{y \in \bound(Q)} ||y - x^u||}, \frac{2}{\min_{v \not = u} ||x^u - x^v||} \right\} \right) = \notag \\
C & \left(\arg \max_{x \in Q} \min_{u \in K} \min \left\{ \min_{y \in \bound(Q)} ||y - x^u||, \frac{\min_{v \not = u} ||x^u - x^v||}{2} \right\} \right).
\label{eq:long}
\end{align}
\caption{The min-max exchange discussed in the first part of the proof of Proposition 3.}
\end{figure*}

\begin{remark}
\label{remark:rate}
Under additional assumptions as per Theorem 3.1 of of \cite{Bugarin2011}, 
which are not too restrictive, but non-trivial to formulate, 
one may prove finite convergence. 
Bounds of the rates of convergence may solve a major open problem in mathematical optimization,
though.
Indeed, notice the connection to the planar geometric packing problem \cite{Fowler1981}
and the 
Min-Max Multicenter \cite[ND50]{Garey1979}, 
which are NP-Complete.
\end{remark}

\begin{remark}
Numerically, the social optimum and the best possible outcome given a price function (\ref{eq:SocialOptBasic}) can be computed using GloptiPoly3 \cite{henrion2009}.
Computing the social optimum for more than 30 cars may be a challenge, though,
considering the general fractional programming involved.
In the special case, where the price function is linear in $x$, it suffices to solve a linear program in order to compute the outcome,
but that seems too restrictive.
\end{remark}

Analytically, it is clear: \redd{[What does the following proposition
  say? Can we explain in English?]}
\begin{proposition}
Under Assumption
\ref{as:Holonomic}, 
for any $U(x, u) \not= U^*(x, u) \quad \forall x \in Q^k$, $u \in K$, we have  $C^* \le C \left( \arg \min_{x} \min_{u \in K} U(x, u) \right)$.
For $V$ in the neighbourhood model, we have  $C^* = C \left( \arg \min_{x} \min_{u \in K} V(x, u) \right)$.
\end{proposition}

\section{An Analysis of the Dynamics in Continuous Time}

The dynamics in continuous time, i.e. 
$x(t)$ generated by \eqref{eq:dynamics} in continuous time,
provide a seemingly very crude approximation of the discrete-time 
sequence $x_n$, 
Nevertheless, in this continous-time approximation, one can show that $x(t)$ 
\eqref{eq:dynamics} converge to the best possible outcome for $V$.

Let us compare the social optimum corresponding to the inconvenience \eqref{eq:pair-wise-sum}
with the fixed point obtained by agents acting rationally to the price function given by density approximation using the closest other agent.
Under some assumptions, which we unnecessarily strengthen in order to avoid rounding:

\begin{assumption}[Square] There exists an integer $i$ such that there are $k = i^2$ agents
and let the region of Assumption \ref{as:Holonomic} be a unit square.
\label{as:Square}
\end{assumption}

\begin{proposition}[Convergence to the Optimum Position]
Under Assumption
\ref{as:Holonomic} for $d = 2$, 
for any $x_0 \in Q$,
there exists a fixed point $x_{\infty}$ for:
\begin{align}
\frac{d}{dt} x(t) & = - \LeastN\left( \partial \left( 
\arg \min_{x \in Q} V(x, u) 
\right) \right) (x(t)). 
\end{align}
Furthermore, under Assumption~\ref{as:Square}, 
$C^* = C(x_{\infty})$, i.e. the cost \eqref{eq:social-cost} of the outcome of the dynamics \eqref{eq:dynamics} 
given by the rational response to the pricing with respect to the neighbourhood-based price $V$,
converges to the fixed-point $x_{\infty}$, whose cost is at the social optimum \eqref{eq:social-opt}, 
for any $|D| \ge 1$ in finite time.
\label{prop:inSquare}
\end{proposition}

\begin{proof}[Proof of Proposition \ref{prop:inSquare}]
The proof has two steps. First, we show the equivalence to another problem, and then we use
the results of Cort{\'e}s and Bullo \cite{Cortes2005} on the equivalent problem.

First, 
we need to show that for distinct positions of agents, we can exchange the min and max
operators, as in \eqref{eq:long}.

Let us start with an easier problem:
\begin{align}
C \left(\arg \min_{x \in Q} \max_{u \not = v \in K} \left\{ \frac{2}{ ||x^u - x^v||} \right\} \right) = \\
C \left(\arg \max_{x \in Q} \min_{u \not = v  \in K} \left\{ \frac{||x^u - x^v||}{2} \right\} \right)
\label{eq:easy}
\end{align}
Let us consider a set $A$ of $k$ distinct vectors in $\RR^2$ and
two distinct vectors $x^u, x^v \in A, x^u \not= x^v$. Clearly:
\begin{align}
\arg \max_{x \in A} || x^u - x^v || = \arg \min_{x \in A} \frac{1}{||x^u - x^v||}
\label{eq:1}
\end{align}
Next, let us consider one of the non-unique:
\begin{align}
(u^*, v^*) \in \arg \min_{u, v \in K, u \not=v} ||x^u - x^v|| \\
           = \arg \max_{u, v \in K, u \not=v} \frac{1}{||x^u - x^v||}
\end{align}
Plug in $(u^*, v^*)$ into \eqref{eq:1} and multiply and divide by a constant as needed. 
Notice that the non-uniqueness stems from the fact the pair of vectors $x^u, x^v$ at the optimum is non-unique,
  but that the distance $||x^u - x^v||$ is the same across the optima, and hence all the optima are equivalent with respect to $C$,
  which considers only the distance $||x^u - x^v||$.
On both left- and right-hand side, we hence obtain optima, which are equivalent with respect to $C$,
as required in \eqref{eq:easy}.
Finally, notice that the extension to the distances to the boundary only enlarges the set of distances, among which to pick the optimum,
where some of the points are projections of vectors $x^u$ onto the boundary and the corresponding 
and the corresponding distances are $\min_{y \in \bound(Q)} ||y - x^u||$.



The rest of the proof is based on the results of Cort{\'e}s and Bullo \cite{Cortes2005} for:
\begin{align}
\arg \max_{x \in Q} \min_{u \in K} \min \left\{ \min_{y \in \bound(Q)} ||y - x^u||, \frac{\min_{v \not = u \in K} ||x^u - x^v||}{2} \right\}.
\end{align}
Let us use the shorthand $\textrm{sm}(u) := \min \left\{ \min_{y \in \bound(Q)} ||y - x^u||, \frac{\min_{v \not = u \in K} ||x^u - x^v||}{2} \right\}$ 
in keeping with the notation of \cite{Cortes2005}.
Notice that the continous-time dynamical system capturing the rational response to 
computing the density based on the distance to the closest car is:

\begin{align}
{\dot x}_i(t) = - \LeastN( \partial \textrm{sm}(u)) (x(t)), 
\end{align}

where $\LeastN$ is a mapping, which associates to each subset $S$ of $R^d$ the an elements of $S$ minimising the norm, as above.
By Theorem 2.7 and Proposition 3.5 of Cort{\'e}s and Bullo \cite{Cortes2005}, the gradient flow corresponding to \eqref{eq:dynamics2} results in the so called ``inscribed circle'' partition, 
whose social cost \eqref{eq:social-cost} is clearly the same as that of the uniform spacing.
\end{proof}

\begin{remark}
Notice that Proposition~\ref{prop:inSquare} holds for any initial distribution of the agents' positions in the square, but the fixed point is not necessarily unique, although all the non-unique fixed points have the same cost  \eqref{eq:social-cost}.
\end{remark}

\begin{remark}
If one were to analyse the discrete-time dynamics \eqref{eq:dynamics2},
one needs to consider the step-size, which makes the analysis considerably more complicated.
Notice that just as in Remark~\ref{remark:rate}, one should not expect strong
results on the rate of convergence, especially for a general $Q$.
\end{remark}

\section{Illustrative Simulations}
\label{sec:experimental}

Further observations can be made using simple discrete-time simulations. 
First, let us consider the consistency of the synchronous approximation 
  of the asynchronous process \eqref{eq:dynamics2}
 in Figure~\ref{fig:uniform}.
We compare the movement of a single car in each step, using functions $p_n$ including:
\begin{itemize}
\item $p^1_n$: pick an element of $\{ 1, 2, \ldots, k \}$ uniformly at random, with no repetion within interval $[ik + 1, (i+1)k]$ for $i = 0, 1, \ldots$ \footnote{
That is:
if $n = ik + 1$ for $i = 0, 1, \ldots$, shuffle $\{ 1, 2, \ldots, k \}$, and fix the order as $O$. In any case, return the element at position $(n \mod k) + 1$ of $O$.
}
\item $p^2_n$: pick an element $\{ 1, 2, \ldots, k \}$ uniformly at random, with repetition
\item $p^3_n := (n \mod k) + 1$, i.e. in in the cyclic order.
\end{itemize} 
with the movement of all vehicles at the same time.
In this example, which uses ``uniformdata'' from Matlab's gallery of matrices,
  the consistency seems near perfect, although we have not provided any guarantees 
  or theoretical justification as to why this should be the case.

Next, let us illustrate the fixed point of the discrete time dynamical system \eqref{eq:dynamics2} described in Proposition~\ref{prop:inSquare} above. 
Notice in Figure~\ref{fig:uniform}, again, that even a small number of steps \eqref{eq:dynamics2}, albeit depending on the number of agents,
suffices to get close to the ``inscribed circle'' partition, which is socially optimal.

\begin{figure*}[tp!]
  \includegraphics[clip=true,width=0.3 \textwidth]{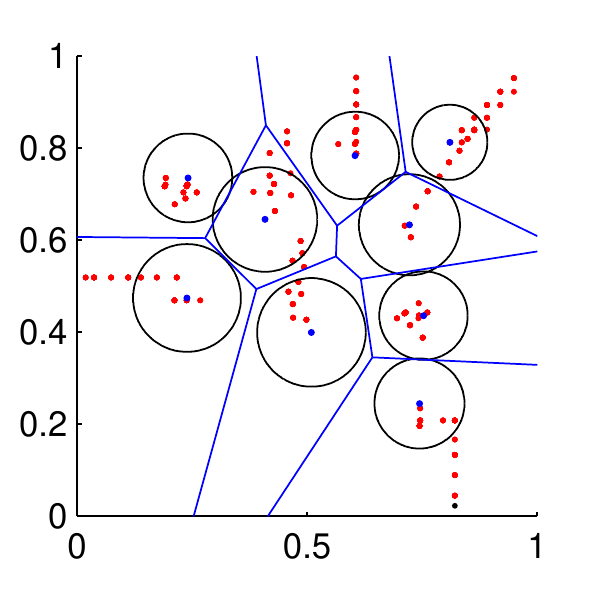}
  \includegraphics[clip=true,width=0.3 \textwidth]{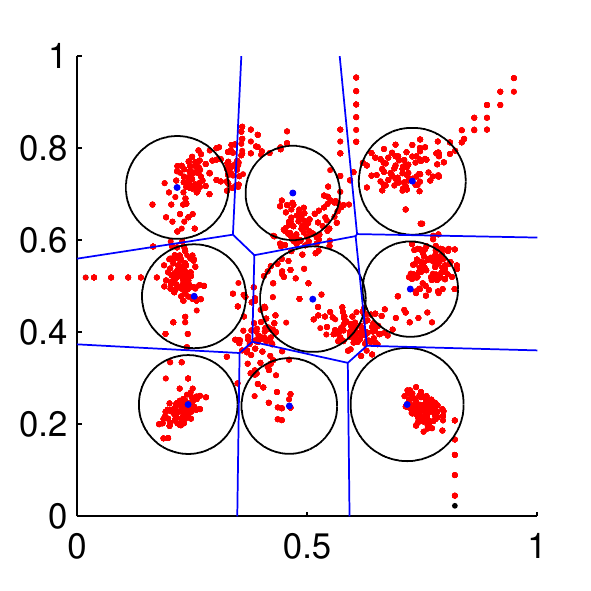}
  \includegraphics[clip=true,width=0.3\textwidth]{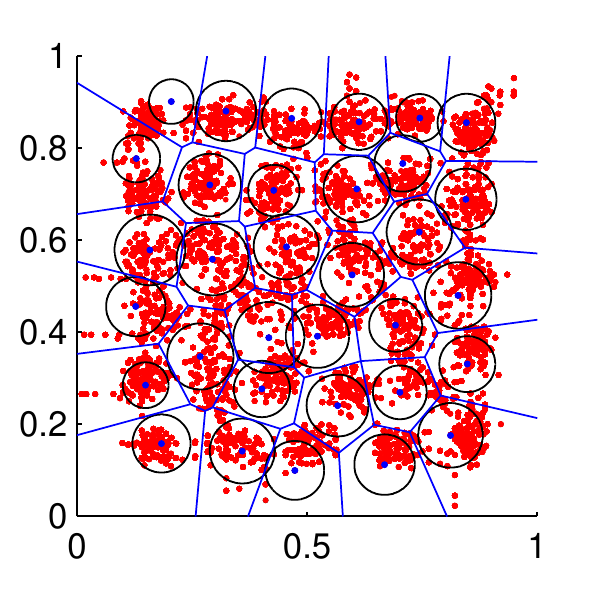}
\\
  \includegraphics[clip=true,width=0.3 \textwidth]{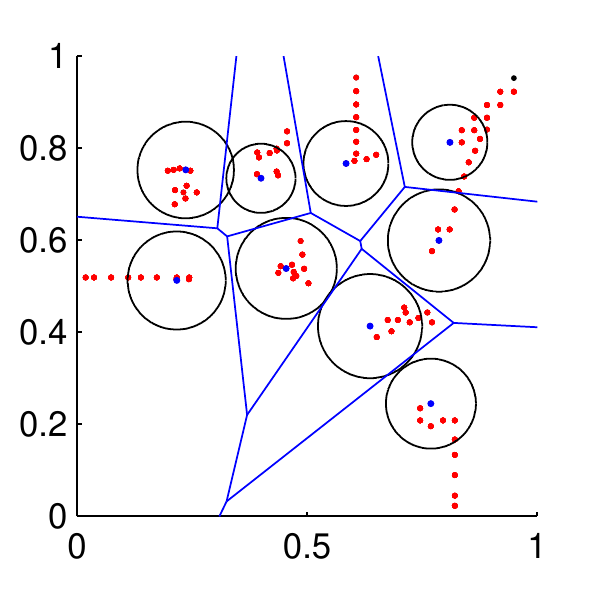}
  \includegraphics[clip=true,width=0.3 \textwidth]{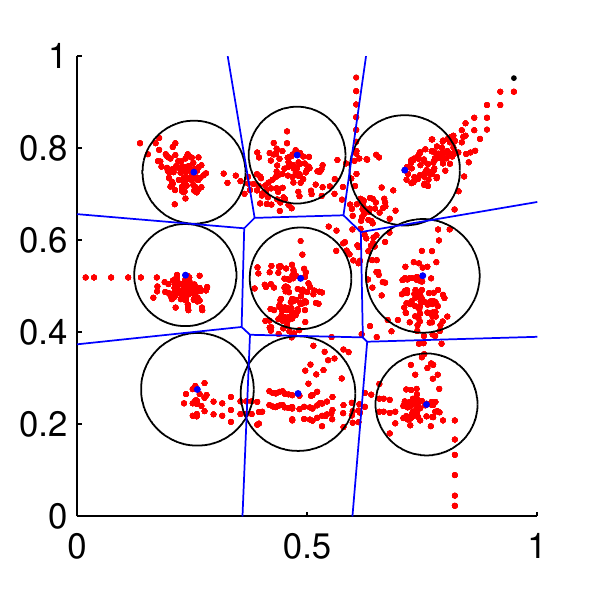}
  \includegraphics[clip=true,width=0.3\textwidth]{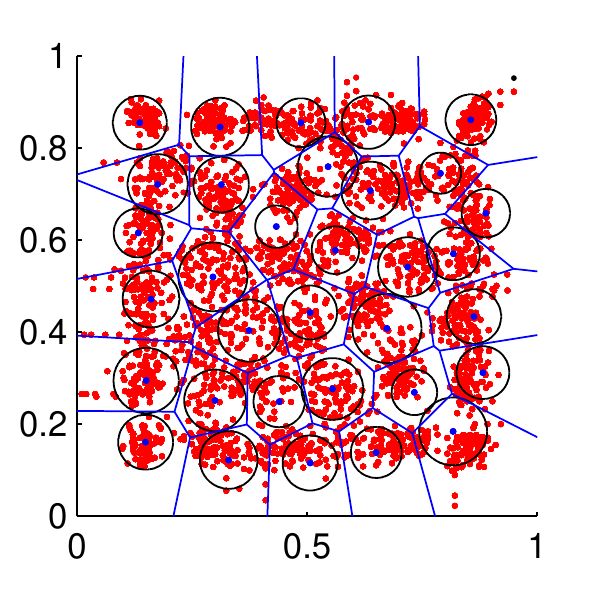}
\\
  \includegraphics[clip=true,width=0.3\textwidth]{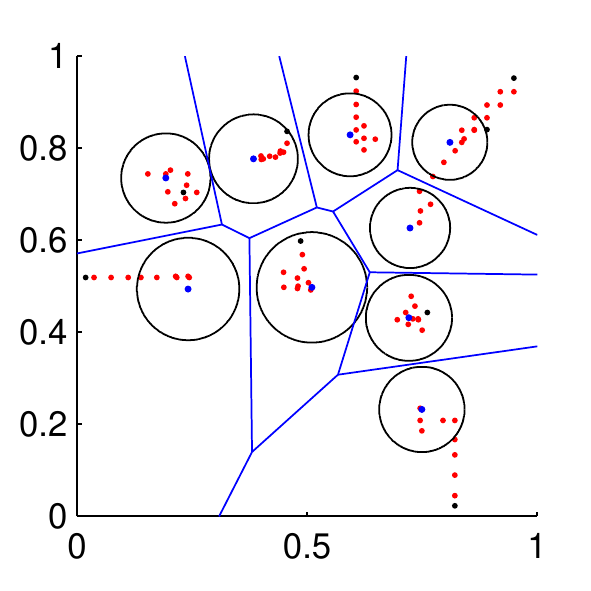}
  \includegraphics[clip=true,width=0.3 \textwidth]{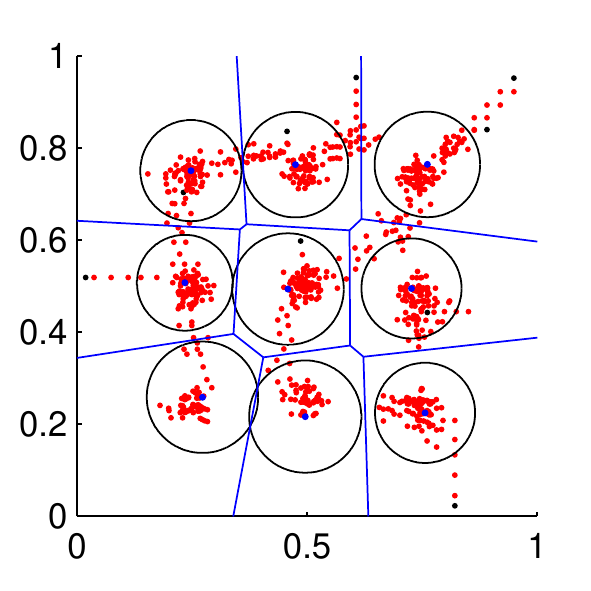}
  \includegraphics[clip=true,width=0.3 \textwidth]{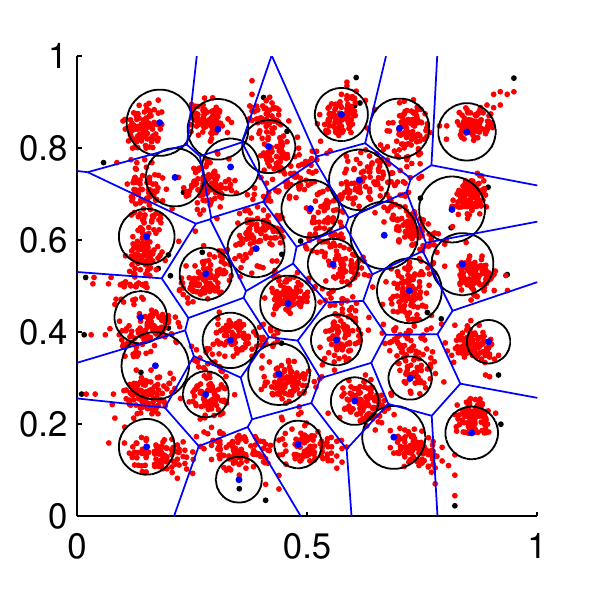}
\caption{The evolution of $x_n$ in a unit square, considering $U^*$ and the bound on the step length $s_u = 0.05$ for all $u \in K$.
The top two rows illustrate the discrete-time dynamics \eqref{eq:dynamics2}, where one driver changes position at any time, according to $p^1$ (top row) or $p^2$ (middle row),
in contrast to the bottom row, where all drivers change position at the same time.
In each row, there are plots
for $k=9$ and $n = 1, 2, \dots 10k$ (left), $k=9$ and $n = 1, 2, \ldots,  100k$ (center),
and $k=33$  and $n = 1, 2, \ldots,  100k$ (right).
In each plot, 
the initial position is marked in black, the final in blue, and positions inbetween in red.
The Voronoi partition corresponding to the final positions is plotted in blue, with the inscribed circles in black.
}
\label{fig:uniform}
\end{figure*}


Next, notice that it is very easy to find examples, where a poor choice of the pricing function 
results in convergence to a fixed point, which is far from social optimum,
or to a non-convergence.
Consider, for instance, Figure~\ref{fig:counterexample}, which again compares the evolution of $x_n$ in a unit square, albeit with 8 instead of 9 points, this time varying the functions used in the neighbourhood-based price 
$V$ and $W$.
Notice that the use of $W$ does not lead to convergence to the socially optimum, even within a unit square, in this example.

\begin{figure*}[tp!]
  \includegraphics[clip=true,width=0.3\textwidth]{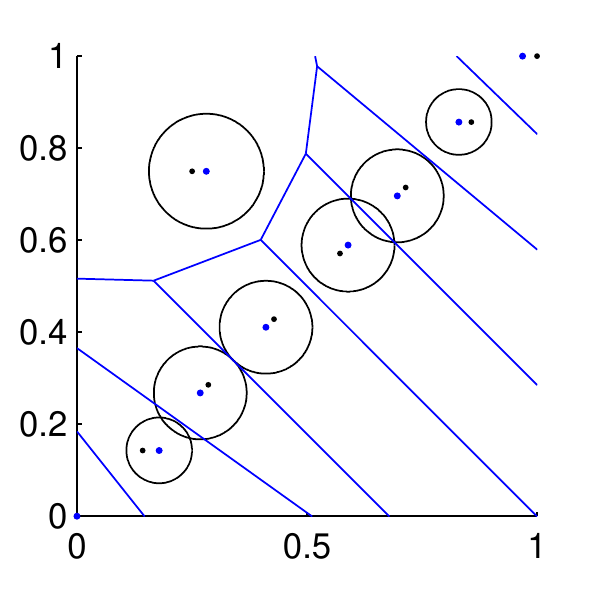}
  \includegraphics[clip=true,width=0.3 \textwidth]{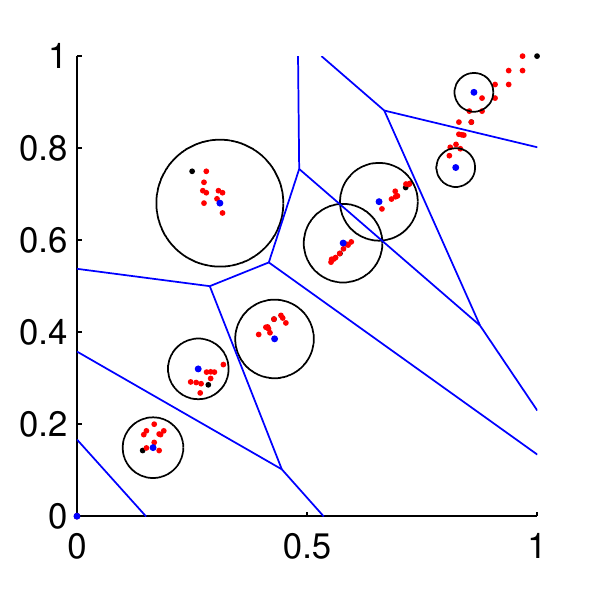}
  \includegraphics[clip=true,width=0.3 \textwidth]{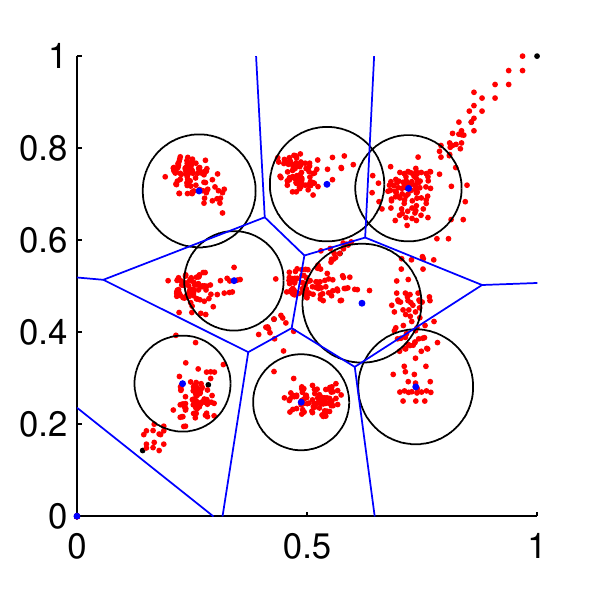}\\
  \includegraphics[clip=true,width=0.3\textwidth]{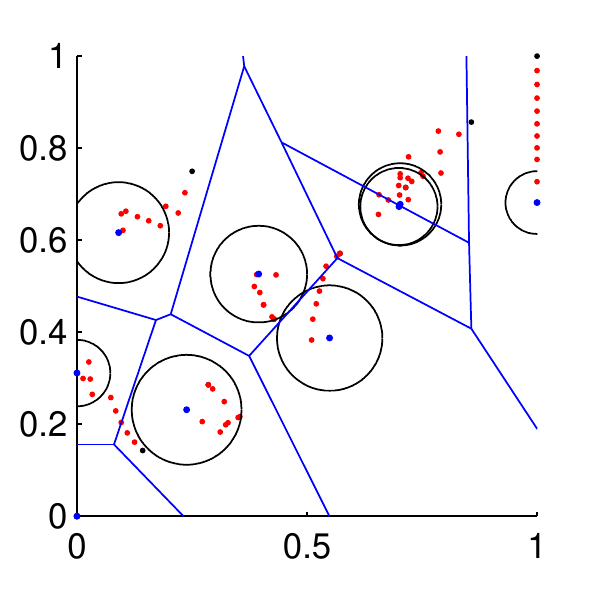}
  \includegraphics[clip=true,width=0.3 \textwidth]{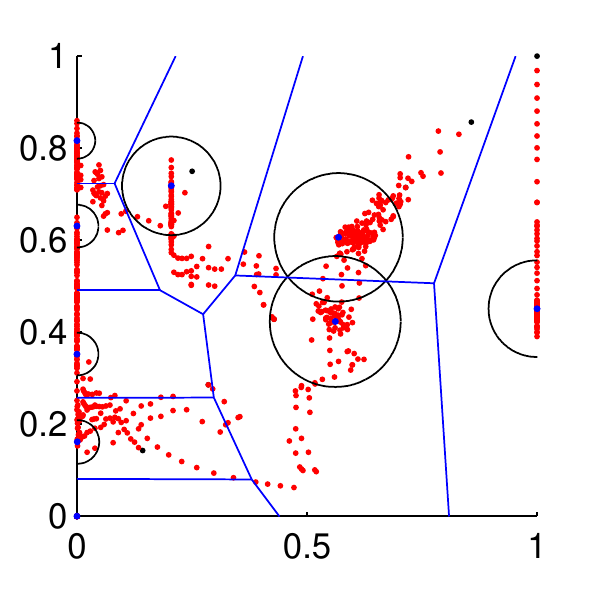}
  \includegraphics[clip=true,width=0.3 \textwidth]{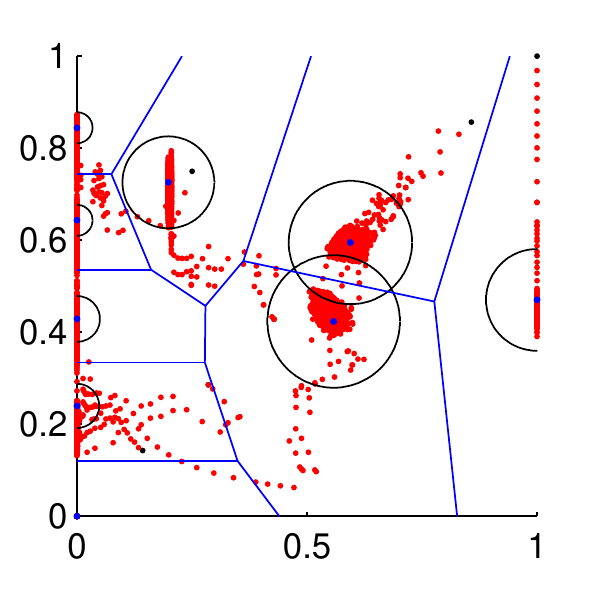}
\caption{
The evolution of $x_n$ in a unit square, considering $V$ (top)
and $W$ in the neighbourhood-based price.
In each row, from left to right, $x_n$ for $n$ equal to $k, 10k$, and $100k$ (top), 
and $n = 10k, 100k$, and $1000k$ steps (bottom), respectively.
In each plot,
we have used $k=9$ and  $s_u = 0.05$ for all $u$.
The initial position is marked in black, the final in blue, and positions inbetween in red.
The Voronoi partition corresponding to the final positions in blue, with the inscribed circles in black.
}
\label{fig:counterexample}
\end{figure*}

\section{Related Work}

The flexible car-sharing systems, allowing for one-way, open-ended journeys are
 rather novel. First mentioned ten years ago \cite{schwieger2004international,millard2005car},
 the first implementations \cite{firnkorn2011will} are only recently evaluated and
 studied in more detail.
The first rigorous study of the related optimisation problems is perhaps due to \cite{Nair2011}.
In contrast to many studies \cite{Nair2011,jorge2013comparing,boyaci2013optimization}, we do not pre-define the locations, and vary the costs.

In the context of bicycle sharing, the problem is somewhat different, but there is a larger
  body of literature available. Specifically, whereas a car can easily be parked anywhere, in principle, the bicycle often needs expensive street furniture \cite{nair2013}, and hence a pre-defined station.
  Further, whereas car-sharing schemes are often operated by profit-seeking companies,
 the bicycle sharing schemes are often run by the authorities, who want to keep prices low, 
 and often nominal.
Two of the most advanced studies \cite{nair2014a,nair2014b} 
 formulate a bi-level optimisation problem
 capturing the equilibrium of the demand and supply,  
 although the exact solution remains a challenge, computationally.
\cite{fricker2012,Fricker2014,fricker2014two} provide elaborate 
 stochastic analyses of the performance of a closed queuing system, 
 associated with bicycle rentals.
They show that without pricing and with capacity limits at the stations, 
 there are a number of stations where either bikes are not available or no further 
 bikes can be dropped off, even in a ``homogeneous'' city, which is well defined and
 in some sense the best possible.
Among other results, they also show that if a user provides his destination and the system forces
 him to take the less full out of two closest stations, the performance improves considerably. 
\cite{waserhole2013pricing} formulate a more elaborate continuous-time 
 Markov chain model, and provide a heuristic for the setting of prices for dropping the bicycle 
 at a given station, specified at pick-up time.
Part V of \cite{Pfrommer2014} suggests model predictive control for the 
 problem, where user arrives at a station to drop off his bicycle, but is offered incentives 
 for dropping the bicycle off at a different station close-by, considering an approximation of 
 the probability user would cycle further in order to gain the incentive by a linear function 
 of the incentives, independent of the distance, and there being a known desirable ``fill level'' at each station.
Incentives have also been used in Paris, where used of V{\'e}lib \cite{nair2013} get credits 
 for leaving the bicycles at up-hill stations.

\section{Conclusions}

We have proposed a pricing scheme for car-sharing applications, which
yields a desirable spread of parked cars throughout the given region,
under rather strong assumptions,
while removing privacy concerns,
inasmuch the system operator needs to know the position of the car only once
it has been dropped off. 
Overall, our analysis and simulations give quantitative guidance on how to
design a readily implementable, privacy-preserving, and efficient pricing scheme.


\bibliographystyle{IEEEtran}
\bibliography{parking}


\end{document}